\documentclass{article}
\usepackage[T1]{fontenc}
\usepackage{graphicx}
\usepackage[utf8]{inputenc}
\usepackage[ams,todos]{CMImacros}
\usepackage{amsmath}
\usepackage{amsfonts}
\usepackage{amssymb}
\usepackage{amsthm}
\usepackage{enumerate}
\usepackage{cleveref}
\usepackage{orcidlink}
\usepackage{placeins} % 
\usepackage[centerlast,margin=.9cm]{caption}
\DeclareCaptionLabelFormat{bold}{\textbf{#1~#2}}
\usepackage{hyperref}
\hypersetup{
	breaklinks=true
}
\usepackage{float}

\usepackage[
  backend=bibtex,
  style=numeric,
  sorting=nyt,    % sort by author, year, title
  maxbibnames=99,
  url=false,
  doi=false
]{biblatex}
\DeclareFieldFormat[article]{title}{#1} % print title literally
\DeclareFieldFormat{title}{#1}
\DeclareFieldFormat{journaltitle}{#1}
\DeclareFieldFormat{booktitle}{#1}
\DeclareFieldFormat{pages}{#1}
\DeclareFieldFormat[inproceedings]{title}{#1}

\DeclareNameAlias{author}{given-family}
\renewcommand*{\labelnamepunct}{\addcolon\space}

\renewbibmacro{in:}{}

\AtEveryBibitem{%
  \clearfield{month}%
  \clearfield{isbn}%
  \clearfield{issn}%
  \clearfield{address}%
  \clearfield{note}%
  \clearfield{language}%
  \ifentrytype{article}{}{%
    \clearfield{publisher}%
  }%
}

\renewbibmacro*{title}{%
  \ifentrytype{article}{%
    \printfield{title}%
  }{%
    \printfield[title]{title}%
  }%
}

\renewbibmacro*{title}{%
  \ifentrytype{thesis}{%
    \printfield{title}%
  }{%
    \printfield[title]{title}%
  }%
}

\DeclareFieldFormat{doi}{%
  \adddot\space
  \url{https://doi.org/#1}%
}

\DeclareBibliographyDriver{article}{%
  \usebibmacro{bibindex}%
  \usebibmacro{begentry}%
  \printnames{author}%
  \setunit{\labelnamepunct}\space
  \usebibmacro{title}%
  \adddot\space
  \printfield{journaltitle}%
  \setunit{\space}%
  \mkbibbold{\printfield{volume}}%
  \setunit{\addcomma\space}%
  \printfield{pages}%
  \setunit{\space}%
  \printtext[parens]{\printfield{year}}%
  \newunit\newblock
  \iffieldundef{doi}{}{%
    \printfield{doi}% <--- MODIFIED
  }%
  \usebibmacro{finentry}%
}

\DeclareBibliographyDriver{book}{%
  \usebibmacro{bibindex}%
  \usebibmacro{begentry}%
  \printnames{author}%
  \setunit{\labelnamepunct}\space
  \printfield{title}%
  \adddot\space
  \printlist{publisher}%
  \setunit{\addcomma\space}%
  \printlist{location}%
  \setunit{\space}%
  \printtext[parens]{\printfield{year}}%
  \usebibmacro{finentry}%
}

\DeclareBibliographyDriver{incollection}{%
  \usebibmacro{bibindex}%
  \usebibmacro{begentry}%
  \printnames{author}%
  \setunit{\labelnamepunct}\space
  \printfield{title}%
  \adddot\space
  \printtext{In: }%
  \printnames{editor}%
  \setunit{\addcomma\space}%
  \printtext{(eds.) }%
  \printfield{booktitle}%
  \setunit{\addcomma\space}%
  \printfield{pages}%
  \adddot\space
  \printlist{publisher}%
  \setunit{\addcomma\space}%
  \printlist{location}%
  \setunit{\space}%
  \printtext[parens]{\printfield{year}}%
  \usebibmacro{finentry}%
}

\DeclareBibliographyDriver{misc}{%
  \usebibmacro{bibindex}%
  \usebibmacro{begentry}%
  \printnames{author}%
  \setunit{\labelnamepunct}\space
  \printfield{title}%
  \adddot\space
  \iffieldundef{journal}{%
    \iffieldundef{eprint}{}{%
      \printtext{arXiv:\space\printfield{eprint}}%
    }%
  }{%
    \printfield{journal}%
  }%
  \setunit{\space}%
  \printtext[parens]{\printfield{year}}%
  \usebibmacro{finentry}%
}

\DeclareBibliographyDriver{inproceedings}{%
  \usebibmacro{bibindex}%
  \usebibmacro{begentry}%
  \printnames{author}%
  \setunit{\labelnamepunct}\space% <--- Uses the new definition (now just a space)
  \printfield{title}%
  \adddot
  \setunit{\addspace}\space
  \printtext{In: }%
  \printfield{booktitle}%
  \setunit{\addcomma\space}%
  \printfield{pages}%
  \setunit{\adddot\space} % Use a period before the publisher/year block
  \printlist{publisher}%
  \setunit{\space}%
  \printtext[parens]{\printfield{year}}%
  \usebibmacro{finentry}%
}

\DeclareFieldFormat[phdthesis]{title}{#1}
\DeclareFieldFormat[thesis]{title}{#1}
\DeclareBibliographyAlias{thesis}{phdthesis}

\DeclareBibliographyDriver{phdthesis}{%
  \usebibmacro{bibindex}%
  \usebibmacro{begentry}%
  \printnames{author}%
  \setunit{\labelnamepunct}\space
  \printfield[titlecase]{title}%
  \adddot\space
  \printfield{type}%
  \space
  \printfield{school}%
  \printtext[parens]{\printfield{year}}%
  \usebibmacro{finentry}%
}

\usepackage{authblk}
\usepackage[margin=1.1in]{geometry} 

\graphicspath{{./figures/}} %
\newcommand{\cfeldesy}{Center for Free-Electron Laser Science CFEL, Deutsches
      Elektronen-Synchrotron DESY, Notkestr. 85, 22607 Hamburg, Germany}%

\newcommand{\uhhmaths}{Department of Mathematics, Universität 
Hamburg, Bundesstr. 55,
      20146, Hamburg, Germany}%

\newtheorem{theorem}{Theorem}[section]
\newtheorem{lemma}{Lemma}[section]

\theoremstyle{definition}

\theoremstyle{example}
\newtheorem{example}{Example}[section]

\theoremstyle{corollary}
\newtheorem{corollary}{Corollary}[section]

\theoremstyle{theorem}
\newtheorem{remark}{Remark}[section]

\numberwithin{equation}{section}

\bibliography{extracted}%

\title{Inducing Riesz Bases in $L^2$ \emph{via} Composition Operators}
% ...
\author[1,2]{Yahya Saleh\thanks{Corresponding Author:
yahya.saleh@uni-hamburg.de} \orcidlink{0000-0002-3235-217X}}
\author[1]{Armin Iske\orcidlink{0000-0003-1743-5484}} 

\affil[1]{\small \uhhmaths}
\affil[2]{\small \cfeldesy}

\date{}
\usepackage[left]{lineno}
\begin{document}
\maketitle
%\linenumbers
%\modulolinenumbers[1] 
\begin{abstract}\noindent%
	Let $C_h$ be a composition operator mapping $L^2(\Omega_1)$ into
	$L^2(\Omega_2)$ for some open sets $\Omega_1, \Omega_2 \subseteq
	\mathbb{R}^n$. We characterize the mappings $h$ that transform Riesz bases
	of $L^2(\Omega_1)$ into Riesz bases of $L^2(\Omega_2)$. Restricting our analysis to differentiable mappings, we
	demonstrate that mappings $h$ that preserve Riesz bases have Jacobian
	determinants that are
	bounded away from zero and infinity. We discuss implications of these results for
approximation theory, highlighting the potential of using
bijective neural networks to construct Riesz bases with favorable
approximation properties. 
\end{abstract}
Frame. 


\section{Introduction}

Consider the measurable spaces $\bigl(\Omega_1, \mathcal{B}(\Omega_1)\bigr)$ and
$\bigl(\Omega_2,\mathcal{B}(\Omega_2)\bigr)$, where $\Omega_1, \Omega_2
\subseteq \mathbb{R}^n$ are open
subsets and $\mathcal{B}(\Omega_1), \ \mathcal{B}(\Omega_2)$
are their respective Borel $\sigma-$algebras. Endow both spaces with the Lebesgue measure $\mu$ and let $h: \Omega_2
\to \Omega_1$ be a measurable, non-singular mapping, \ie, $\mu(B)=0$ implies $\mu(h^{-1}(B))=0$ for any $ B \in
\mathcal{B}(\Omega_1)$. $h$ induces a linear operator 
$$
C_h: L^2(\Omega_1) \to L(\Omega_2), \quad C_h f \mapsto f \circ h,
$$
where $L(\Omega_2)$ denotes the space of all measurable functions on $\Omega_2$.
If \(C_h\) acts boundedly from \(L^2(\Omega_1)\) to \(L^2(\Omega_2)\), we say
that \(C_h\) maps \(L^2(\Omega_1)\) into \(L^2(\Omega_2)\) and refer to it as a
\emph{composition operator from \(L^2(\Omega_1)\) into \(L^2(\Omega_2)\)}.

The study of composition operators has a long history in functional
and operator analysis, particularly on spaces of analytic
functions~\cite{Cowen:CompOp20:1995}, as well as on Hardy and Bergman
spaces~\cite{Nordgren:Composition1978}.
Early work by \textcite{Singh:CompOp1993} explored their boundedness, invertibility,
and compactness in  $L^p$  spaces, while \textcite{Takagi:FS232:321} extended
some of
these results to operators mapping between different  $L^p$  spaces. These
foundational studies have laid the groundwork for further investigations into
the structural and functional properties of composition operators in broader settings.

Let
$(\gamma_n)_{n=0}^\infty$ be a Riesz basis of $L^2(\Omega_1)$.
In this article, we investigate the conditions under which the sequence 
\begin{equation}
	\label{eq:perturbed_comp}
	(C_h \gamma_n)_{n=0}^\infty  
\end{equation}
forms a Riesz basis of $L^2(\Omega_2)$ and characterize its dual, see \Cref{theorem:induced_Riesz} and
\Cref{theorem:dual_Riesz}, respectively. 
Given the potential importance of such induced sequences in optimization
problems, we restrict our analysis to the case where $h$ is a differentiable
mapping. Our results show that all Riesz bases of the form~\eqref{eq:perturbed_comp} 
are induced precisely by mappings with finite volume distortion; that is, 
by mappings whose Jacobian determinants are bounded away from zero and infinity, see~\Cref{theorem:induced_onbases_diff}.  To illustrate the practical implications
of our study, we
present a simple
numerical example demonstrating how the mapping $h$ can be modelled by an
invertible neural network to construct Riesz bases of the form
\eqref{eq:perturbed_comp} with favorable approximation properties.

Our study is motivated by recent computational advances in nuclear-motion theory~\cite{Saleh:JCTC21:5221, Vogt:JCP163:15,Zhang:JCP161:024103} 
and condensed matter physics~\cite{Zhang:PRL143:246101,Xie:arXiv2105.08644},
where parameterized and optimized composition operators have been used to 
construct bases with improved approximation properties. While similar problems 
have been explored in settings such as weighted Hilbert Bergman 
spaces~\cite{Manhas:IEOT91:34}, spaces of holomorphic functions~\cite{Manhas:CAOT18:34}, 
weighted Hardy spaces in the unit disk~\cite{DongL:AMS40:1645}, and Fock spaces~\cite{Liang:RM79:127}, a unified perspective 
for $L^2$ spaces remains unexplored. To address this gap, we rely on Singh’s characterization of bounded and invertible 
composition operators~\cite{Singh:BAMS19:81}, which extends naturally to 
composition operators between \(L^2\) spaces on distinct domains.

\section{Preliminaries}
\label{sec:preliminaries}

A natural measure on $\Omega_1$ is the push-forward measure $h_\# \mu$, defined by
$$
h_\# \mu(B) := \mu(h^{-1}(B)), \quad \text{for any } B \in \mathcal{B}(\Omega_1).
$$
Since $h$ is assumed to be non-singular, $h_\# \mu \ll
\mu$, \ie, $h_\# \mu$ is absolutely continuous
with respect to $\mu$. We denote by
$g_h$ the Radon-Nikodym derivative of $h_\# \mu$ with respect to
$\mu$. 

Throughout the manuscript, we call the mapping $h$ injective if there exists a
measurable transformation $h_1 : \Omega_1 \to \Omega_2$ such that $h_1 \circ h$
coincides with the identity mapping on $\Omega_2$ up to a set of $\mu$-measure zero.
$h_1$ is called the left inverse of $h$.
Similarly, $h$ is called surjective if there exists a measurable transformation
$h_2 : \Omega_1 \to \Omega_2$ such that $h \circ h_2$ coincides with the
identity mapping on $\Omega_1$ up to a set of $\mu$-measure zero.  $h_2$ is called the
right inverse of $h$.
Finally, $h$ is said to be bijective if there exists a measurable mapping $h^{-1} : \Omega_1 \to \Omega_2$ satisfying
$$
    h^{-1} \circ h (x) = x \quad \text{for a.e. } x \in \Omega_2, \ \text{and } 
    h \circ h^{-1} (y) = y \quad \text{for a.e. } y \in \Omega_1.
$$

Necessary and sufficient conditions for a non-singular mapping $h$ to induce a
composition operator from $L^2(\Omega_1)$ into $L^2(\Omega_2)$ were
characterized by Singh and Manhas~\cite{Singh:CompOp1993} for the special case $\Omega_1 =
\Omega_2$. Later, Takagi and Yakouchi~\cite{Takagi:FS232:321} noted that the same
result holds for the more general case $\Omega_1 \neq \Omega_2$.

\begin{theorem}
	\label{theorem:composition_op}
	
	A non-singular mapping $h: \Omega_2 \to \Omega_1$ induces a
	composition operator from
	$L^2(\Omega_1)$ into $L^2(\Omega_2)$ if and only if $g_h$
	is bounded $\mu-$a.e. on $\Omega_1$. In this case, the norm of $C_h$ is given by
	$$
	\left\|C_h\right\| =  \left\| g_h \right\|_{L^\infty(\Omega_1)}^{1/2}.
	$$
\end{theorem}
\begin{proof}
	See~\cite[Corollary 2.1.2]{Singh:CompOp1993}.
\end{proof}

Without loss of generality, we assume that all functions considered in this work
are real-valued.
The main question of this work is whether the induced sequence
$(\phi_n)_{n=0}^\infty := (C_h \gamma_n)_{n=0}^\infty$ is complete in $L^2(\Omega_2)$,
i.e., whether there exists for any $f$ in $L^2(\Omega_2)$ a sequence of coefficients
$(c_n(f))_{n=0}^\infty$, $c_n \in \mathbb{R}$ for any $n \in \mathbb{N}$,  
such that 
$$
f = \sum_{n=0}^\infty  c_n(f) \phi_n.
$$
We discuss in our work
two notions of completeness. First, that of Riesz bases. Here,
the coefficients are unique, and are given by the inner product of $f$
with another Riesz basis $(\widetilde{\phi}_n)_{n=0}^\infty$ that is 
bi-orthogonal to the
original Riesz basis, \ie, 
$$
\langle \widetilde{\phi}_n, \phi_m \rangle_{\Omega_2} = \delta_{nm} \quad \text{for any 
}n,m\in \mathbb{N},
$$
where $\langle., .\rangle_{\Omega_2}$ denotes the inner product on $L^2(\Omega_2)$.
Second, that of orthonormal bases. Here, the
coefficients are unique and given by the inner product of $f$ with the basis
functions. 
%Third, that of frames, where the
%coefficients are not necessarily unique.

 An important result for our analysis is the fact that these complete sequences 
can be
characterized by operators acting on an orthonormal 
basis~\cite{Christensen:IntroFrames2016}. In particular, let
$(\gamma_n)_{n=0}^\infty$ be an orthonormal basis of $L^2(\Omega_1)$. Then 
\begin{itemize}

	\item The Riesz bases of $L^2(\Omega_2)$ are precisely the sequences $(U
	\gamma_n)_{n=0}^\infty$, where $U:L^2 (\Omega_1)\to L^2(\Omega_2)$ is a bounded bijective
	operator. 
	\item The orthonormal bases of $L^2(\Omega_2)$ are precisely the sequences $(U
	\gamma_n)_{n=0}^\infty$, where $U:L^2(\Omega_1) \to L^2(\Omega_2)$ is a unitary operator.  
\end{itemize}

\section{Inducing Bases \emph{via} Composition Operators}
\label{sec:main_results}

We begin by citing a key result for our analysis. Although this result was
	originally proven for the specific case $\Omega_1 = \Omega_2$, it is valid for the more general case $\Omega_1
	\neq \Omega_2$ using the same proof.

\begin{lemma}
	\label{lemma:inv_comp}
	Let 
	$C_h$ be a composition operator from $L^2(\Omega_1)$ into $L^2(\Omega_2)$. If $C_h$ is
	invertible, its inverse is a composition operator from
	$L^2(\Omega_2)$ into $L^2(\Omega_1)$.  
	
\end{lemma}
\begin{proof}
Since $C_h$ is invertible, it follows that $g:=C_h f$ is a characteristic
function if and only if $f$ is a characteristic function, see~\cite[Corollary 2.2.3]{Singh:CompOp1993}. Denote by $A$ the
inverse of $C_h$. It follows immediately that $Ag$ is a characteristic function
if and only if $f$ is a characteristic function.
 	Since $A$ preserves $L^2$ characteristic functions, and $\Omega_1$ and $\Omega_2$ are standard Borel 
	spaces, one can deduce that $A$ is a composition operator, see~\cite[Theorem 
	2.1.13]{Singh:CompOp1993}.
\end{proof}

We now characterize conditions on $h$ for the induced sequence 
\eqref{eq:perturbed_comp} to form a Riesz basis.
\begin{theorem}[Induced Riesz Basis]
	\label{theorem:induced_Riesz}
	Let $(\gamma_n)_{n=0}^\infty$ be a Riesz basis of $L^2(\Omega_1)$. Let $h: 
	\Omega_2 \to
	\Omega_1$ be a non-singular measurable mapping that induces a composition operator $C_h$
	from $L^2(\Omega_1)$ into $L^2(\Omega_2)$, \ie, such that $g_h$ is bounded
	$\mu-$a.e. on $\Omega_1$. The sequence $(C_h \gamma_n)_{n=0}^\infty$ is a 
	Riesz
	basis of $L^2(\Omega_2)$ if and only if $h$ is injective and there exists $r>0$ such that 
	$g_h \geq r$ $\mu-$a.e. on $\Omega_1$.
\end{theorem}
\begin{proof}
	Let $U:L^2 (\Omega_1) \to L^2(\Omega_1)$ be a unitary operator such that
	$(\gamma_n)_{n=0}^\infty =  (U\phi_n)_{n=0}^\infty$ for some orthonormal
	basis $(\phi_n)_{n=0}^\infty$ of $L^2(\Omega_1)$.

	Let $(C_h \gamma_n)_{n=0}^\infty$ be a Riesz basis of $L^2(\Omega_2)$. This implies
	that $C_h U$ is bijective. Since $U$ is unitary, it immediately follows that
	$C_h = C_h U U^{-1}$ is bijective, where $U^{-1}$ denotes the inverse of $U$.
	Denote the inverse of $C_h$ by by $C^{-1}_h$. From
	\Cref{lemma:inv_comp} it follows that $C^{-1}_h: L^2(\Omega_2)
	\to L^2(\Omega_1)$ is a composition operator. There exists a mapping 
	$\zeta$ such that $C^{-1}_h =
	C_\zeta$. Since $ C_h C_\zeta f = f$ for any $f$ in $L^2(\Omega_2)$ and
	$C_\zeta C_h g= g$ for any $g$ in $L^2(\Omega_1)$, it follows that
$h$ is bijective with
	inverse $h^{-1} = \zeta$.  
	For any $B \in \mathcal{B}(\Omega_2)$ denote by $\chi_B$ the characteristic function 
	and observe that 
	\begin{align*}
		\mu(B) &= \int_{\Omega_2} \chi_B \ d\mu \\
		&= \int_{\Omega_2} \chi_B \circ h^{-1} \circ h\ d\mu \\
		&= \int_{\Omega_1} \chi_B \circ h^{-1} \ g_{h} \ d\mu \\
		&\leq \|g_{h}\|_{L^\infty(\Omega_1)} \ h_\#^{-1} \mu (B)\\
		&\leq \|g_{h}\|_{L^\infty(\Omega_1)} \ \|g_{h^{-1}}\|_{L^\infty(\Omega_2)}  \mu (B),
	\end{align*}
	where $g_{h^{-1}}$ denotes the Radon-Nikodym derivative of the measure
	$h_\#^{-1} \mu$. Since $C_{h^{-1}}$ is a composition operator from
	$L^2(\Omega_2)$ into $L^2(\Omega_1)$, it follows that
	$\|g_{h^{-1}}\|_{L^\infty(\Omega_2)}< \infty$. Further, since $C_{h^{-1}}$
	is injective, $g_{h^{-1}}>0$ $\mu$-a.e. on $\Omega_2$. Set $r = 1/\|g_{h^{-1}}\|_{L^\infty(\Omega_2)}$. It
 follows immediately that $ g_h \geq r$ $\mu-$a.e. on $\Omega_1$.
	
	Now suppose that $h$ is injective and $g_h \geq r$ $\mu-$a.e on $\Omega_1$ for 
	some $r
	>0$. Since $h$ is injective, its preimage preserves measurable sets, that is,
$h^{-1}(B) = B$ for all $B \in \mathcal{B}(\Omega_1)$. As shown in~\cite[Theorem~2.5]{Singh:BAMS19:81},
this is necessary and sufficient for the associated composition operator
$C_h$ to have a dense range in $L^2(\Omega_2)$. Together with the lower 
	bound on
	$g_h$, this implies that $C_h$ is surjective, see~\cite[Theorem
	2.3]{Singh:BAMS19:81}. To see that $C_h$ is injective, we observe that for any $B \in \mathcal{B}(\Omega_1)$ we have
	\begin{align*}
		h_\#\mu(B)  &= \int_{\Omega_1} \chi_B \ g_h \ d\mu \\
		&\geq r \ \mu (B).
	\end{align*}
	It follows that $\mu \ll h_\# \mu$. This immediately implies that $C_h$ 
	is injective.
	Since $C_h$ is also bounded we conclude that $C_h U$ is a bounded bijection.
Consequently,	$(C_h 
	\gamma_n)_{n=0}^\infty$ is a Riesz basis of $L^2(\Omega_2)$.
\end{proof}

\begin{remark} 
	Combining \Cref{theorem:composition_op} and 
	\Cref{theorem:induced_Riesz}
	demonstrates that all Riesz bases of the form $(C_h 
	\gamma_n)_{n=0}^\infty$ are
	induced by mappings $h$ that satisfy the following. 
	\begin{enumerate}[(i)]
		\item $h$ is a non-singular injective mapping.
		\item there exist $r, R >0$ such that $r \leq g_h \leq R$ $\mu-$a.e. on $\Omega_1$.
	\end{enumerate}
We observe that these conditions on $h$ ensure its bijectivity. 
\end{remark}

As we shall see, a Riesz basis of the form \eqref{eq:perturbed_comp} and its
dual are related by a multiplication operator. We now recall the necessary
theory to discuss this relationship.

A measurable mapping $k:\Omega_2 \to \Omega_2$ induces a linear operator $M_k$
defined by
$$
M_k: L^2(\Omega_2) \to L(\Omega_2), \quad M_k f(x) = k(x) f(x)
$$
for any function $f$ in $L^2(\Omega_2)$ and any $x$ in $\Omega_2$. If $M_k$ maps
$L^2(\Omega_2)$ into $L^2(\Omega_2)$, it is called a \emph{multiplication
operator from $L^2(\Omega_2)$ into itself}. Necessary and
suﬃcient condition for $M_k$ to map $L^2(\Omega_2)$ into itself is that $k$ is essentially bounded~\cite{Takagi:FS232:321}.

Let $h: \Omega_2 \to \Omega_1$ be a measurable mapping that induces a Riesz
basis $(C_h \gamma_n)_{n=0}^\infty$ of $L^2(\Omega_2)$. In such a case $h$ is
invertible and its inverse $h^{-1}$ defines the push-forward measure $h^{-1}_\#
\mu$. Denote by $g_{h^{-1}}$ the Radon-Nikodym derivative of $h^{-1}_\# \mu$
with respect to $\mu$ and note that $g_{h^{-1}} \in L^\infty (\Omega_2)$. It
follows that $M_{g_{h^{-1}}}$ is a bounded multiplication operator from
$L^2(\Omega_2)$ into itself.

\begin{theorem}
	\label{theorem:dual_Riesz}
	Let $(\gamma_n)_{n=0}^\infty$ be a Riesz basis of $L^2(\Omega_1)$ and denote
	by $(\tilde{\gamma}_n)_{n=0}^\infty$ its dual. Let 
	$h:\Omega_2 \to \Omega_1$
	be a non-singular measurable mapping such that
	$(C_h \gamma_n)_{n=0}^\infty$ is a Riesz basis of $L^2(\Omega_2)$. Then $(M_{g_{h^{-1}}} C_h
	\tilde{\gamma}_n)_{n=0}^\infty$ is the bi-orthogonal dual of $(C_h
	\gamma_n)_{n=0}^\infty$.
\end{theorem}
\begin{proof}
	To prove that the sequence $(M_{g_{h^{-1}}} C_h \tilde{\gamma}_n)_{n=0}^\infty$ 
	is a
	Riesz basis it suffices to show that the operator
	$M_{g_{h^{-1}}}C_h$ is a bounded bijection. $C_h$ acts boundedly from $L^2(\Omega_1)$ into $L^2(\Omega_2)$.
	Moreover, $(C_h \gamma_n)_{n=0}^\infty$ being a Riesz basis implies that
	$C_h$ is a bijection with inverse $C_{h^{-1}}$. This, in turn, implies that
	$h^{-1}$ is non-singular and $g_{h^{-1}}$ is
	bounded $\mu-$a.e. on $\Omega_2$, see \Cref{lemma:inv_comp} and 
	\Cref{theorem:composition_op}. By
	\Cref{theorem:induced_Riesz} it holds that $g_{h^{-1}}$ is bounded away 
	from
	zero for a.e. $x \in \Omega_2$. By combining these results we deduce that 
	$M_{g_{h^{-1}}}$ is a bounded
	bijection. Therefore, $M_{g_{h^{-1}}} C_h$ is a bounded bijection.
	
	To see that $(C_h \gamma_n)_{n=0}^\infty$ and $(M_{g_{h^{-1}}} C_h
	\tilde{\gamma}_n)_{n=0}^\infty$ are bi-orthogonal observe that for any $n, m \in 
	\mathbb{N}$ we have 
	\begin{align*}
		\int_{\Omega_2} (\gamma_n \circ h)\ (\tilde{\gamma}_m \circ h)\ g_{h^{-1}} \ d\mu 
		&= \int_{\Omega_2} (\gamma_n \circ h) \ (\tilde{\gamma}_m \circ h) \ dh^{-1}_\# 
		\mu \\
		&= \int_{\Omega_1} (\gamma_n \circ h \circ h^{-1}) \ (\tilde{\gamma}_m \circ h 
		\circ h^{-1}) \ d\mu \\
		&= \delta_{nm}.
	\end{align*}
	It remains to show that the pair $(C_h \gamma_n)_{n=0}^\infty$ and
	$(M_{g_{h^{-1}}} C_h \tilde{\gamma}_n)_{n=0}^\infty$ are dual. In other words, it
	remains to show that for any $f$ in $L^2(\Omega_2)$ we have
	$$
	f = \sum_{n=0}^\infty \langle f, M_{g_{h^{-1}}} C_h \tilde{\gamma}_n \rangle_{\Omega_2} \ C_h 
	\gamma_n.
	$$
	Note that $C_{h^{-1}} f \in L^2(\Omega_1)$ for any $f \in L^2(\Omega_2)$. Since
	$(\gamma_n)_{n=0}^\infty$ is a Riesz basis of $L^2(\Omega_1)$ it follows that 
	\begin{align*}
		C_{h^{-1}} f &= \sum_{n=0}^\infty \ \langle C_{h^{-1}} f, \tilde{\gamma}_n 
		\rangle_{\Omega_1} \ \gamma_n \\
		&= \sum_{n=0}^\infty \ \biggl(\int_{\Omega_1} f \circ h^{-1} \ \tilde{\gamma}_n \ d 
		\mu \biggr)\ \gamma_n \\
		&= \sum_{n=0}^\infty \ \biggl(\int_{\Omega_2} f \ \tilde{\gamma}_n \circ h \ 
		g_{h^{-1}} \ d \mu \biggr)\ \gamma_n \\
		&= \sum_{n=0}^\infty \ \langle f, M_{g_{h^{-1}}} C_h \tilde{\gamma}_n \rangle_{\Omega_2}\ 
		\gamma_n.
	\end{align*}
	The result follows by applying $C_h$ to both sides.
\end{proof}

For the completeness of our results, we conclude this section by remarks on the possibility of
generating orthonormal bases and frames of the form \eqref{eq:perturbed_comp}.

\begin{theorem}[Induced Orthonormal Bases]
	\label{theorem:induced_onbases}
	Let $(\gamma_n)_{n=0}^\infty$ be an orthonormal basis of $L^2(\Omega_1)$. Let $h: 
	\Omega_2 \to
	\Omega_1$ be a non-singular measurable mapping that induces a composition operator 
	$C_h$
	from $L^2(\Omega_1)$ into $L^2(\Omega_2)$, \ie, such that $g_h$ is bounded
	$\mu-$a.e. on $\Omega_1$.
	The sequence $(C_h \gamma_n)_{n=0}^\infty$ is an
	orthonormal basis of $L^2(\Omega_2)$ if and only if $g_h=1$ $\mu$-a.e. on $\Omega_1$.
\end{theorem}
\begin{proof}Trivial.
\end{proof}

\Cref{theorem:induced_onbases} demonstrates that non-trivial orthonormal bases
cannot be generated by composition operators alone.
Essentially, the problem here is the difficulty of preserving the
orthonormality of the underlying basis. To see this, note, \eg, that for any $m 
\in \mathbb{N}$, 
$$
\int_{\Omega_2} (C_h\gamma_m)^2 \ d\mu = \int_{\Omega_1} \gamma_m^2\ g_h \ 
d\mu.
$$
Requiring that this equals 1 indeed poses a severe restriction on $h$,
namely that $g_h=1$ for a.e. $x \in \Omega_1$. Generating orthonormal bases of
$L^2(\Omega_2)$ can be achieved by using a different approach, such as a weighted
composition operator. To this end, assume that $h: \Omega_2 \to \Omega_1$ is
non-singular and bijective. Since it is bijective, it follows that $h(B) \in
\mathcal{B}(\Omega_1)$ for any $B$ in $\mathcal{B}(\Omega_2)$~\cite[Corollary
15.2]{Kechris:SetTh1:2012}. Note that $h$ induces the pullback measure $h^\# \mu$
on $\Omega_2$ defined by $h^\# \mu( B) := \mu(h(B))$ for any $B \in
\mathcal{B}(\Omega_2)$. Denote by $k_h$ the Radon-Nikodym derivative of $h^\#
\mu$ with respect to $\mu$ and note that the operator 
$$
W_h f: L^2(\Omega_1) \to L(\Omega_2), \quad W_h := (f \circ h) \cdot k^{1/2}_{h}
$$
maps $L^2(\Omega_1)$ into $L^2(\Omega_2)$. Moreover, $W_h$ is unitary and hence
it follows that $(W_h \gamma_n)_{n=0}^\infty$ is an orthonormal basis of
$L^2(\Omega_2)$.

We note that sequences of the form
$(W_h \gamma_n)_{n=0}^\infty$ appeared in some previous studies~\cite{Saleh:PAMM23:e202200239}. However, these
results assumed more regularity on $h$.
\begin{remark}[Flow-Induced Frames] Another important type of complete
	sequences in $L^2$ is that of frames. In contrast to orthonormal and Riesz 
	bases, expansion coefficients of frames are not necessarily unique.
	One might wonder what the precise conditions on $h$ are for the induced 
	sequence \eqref{eq:perturbed_comp} 
	to
	form a frame.     
	It can be demonstrated that frames of the form $(C_h \gamma_n)_{n=0}
	^\infty$ are induced precisely by non-singular mappings $h$ that are 
	injective and such
	that $g_h$ is bounded away from zero almost
	everywhere on its support. Deriving a canonical dual frame for such a frame
	is, however, not as straightforward as in the case of Riesz bases of the form 
	\eqref{eq:perturbed_comp}, since $h$ is not necessarily bijective.
\end{remark}

\section{Induced Bases \emph{via} 
Differentiable Mappings}
\label{sec:relaxed_results}
Bases of the form \eqref{eq:perturbed_comp} are of interest in
optimization problems, particularly those involving differential operators. 
To facilitate the use of such bases in applications, we state in
this section the necessary and sufficient conditions for completeness in $L^2(\Omega_2)$
when $h$ is a differentiable mapping.

Under the conditions of \Cref{theorem:composition_op} it immediately follows
that $h_\# \mu$ is locally finite. Hence, assuming that $h$ is everywhere 
differentiable, it follows by
Lebesgue differentiation theorem~\cite{Schilling:MeasureTheory2017} that 
$$
g_h = \frac{1}{|\det J_h| \circ h^{-1}} \qquad \mu-\text{a.e. on } \Omega_1,
$$
where $J_h$ denotes the Jacobian of $h$.
Using the inverse function
theorem, we have 
$$
\frac{1}{|\det J_h| \circ h^{-1}} = |\det J_{h^{-1}}| \qquad 
\mu-\text{a.e. on } \Omega_1.
$$

Based on the gained regularity we can now state the following. 
\begin{corollary}
	\label{theorem:induced_onbases_diff}
	Let $(\gamma_n)_{n=0}^\infty$ be a Riesz basis of $L^2(\Omega_1)$ and let $h: 
	\Omega_2
	\to \Omega_1$ be an everywhere differentiable and injective mapping that induces a 
	composition
	operator from $L^2(\Omega_1)$ into $L^2(\Omega_2)$.
	Then $(C_h \gamma_n)_{n={0}}^\infty$ is a Riesz basis of $L^2(\Omega_2)$ if and only
	if $r \leq \det J_h \leq R$ for a.e. $x$ in $\Omega_2$ and some $r, R >0$.
\end{corollary}
\begin{proof}
	Note that when $h$ is differentiable and injective, the lower and upper bound on $\det 
	J_h$
	imply that $h$ is bijective. Therefore, the result follows immediately from 
	\Cref{theorem:induced_Riesz},
	 and the application of the inverse
	function theorem.
\end{proof}
Observe that in the case where $(C_h \gamma_n)_{n=0}^\infty$ is a Riesz 
basis, its dual is
given by 
$$(M_{\det J_h} C_h \tilde{\gamma}_n)_{n=0}^\infty = (\tilde{\gamma}_n \circ h \ \det
J_h)_{n=0}^\infty,$$
where $(\tilde{\gamma}_n)_{n=0}^{\infty}$ denotes the dual of
$(\gamma_n)_{n=0}^{\infty}$.

In one dimension, the condition $r \leq \det J_h \leq R$ a.e. on $\Omega_2$ simplifies to 
$r \leq h' \leq R$, which means that $h$ is $(r, R)$-bi-Lipschitz, \ie,
\[
r\,|x - y| \leq |h(x) - h(y)| \leq R\,|x - y|, 
\qquad \text{ for a.e. }\, x, y \in \Omega_2.
\]

\begin{remark}
	Often,
	continuous differentiability is assumed to use the inverse function theorem.
	However, differentiability everywhere is enough, 
	see~\cite[Theorem 1]{Raymond:Math49:141}.
\end{remark}
\section{Application: learning a basis with improved approximation properties}
\label{sec:summary}
The use of bases for approximation problems is ubiquitous in
applied and engineering sciences. Typically, an unknown function $f$ on a domain
$\Omega \subseteq \mathbb{R}^n$ is 
approximated
in the linear span of a truncated basis 
$(\gamma_n)_{n=0}^{N-1}$ of the space $L^2(\Omega)$ for
some $N \in \mathbb{N}$. Due to the completeness of $(\gamma_n)_{n=0}^\infty$ in
$L^2(\Omega)$, 
the approximation
error converges to zero as $N \to \infty$. In practice, the accuracy of
the approximation and the convergence rate highly depend on the choice of the 
basis~\cite{Gottlieb:SpectralMethods:1977}. This dependence renders
methods to choose an optimal basis for a given problem 
highly desirable.

We argue here that one can construct a problem-specific basis by
optimizing a perturbation to an orthonormal basis \emph{via} a
composition operator. The following example suggests that a perturbation 
induced
by a function as
simple as a shifting mapping can lead to a significant improvement in the 
approximation of
a target function.

\begin{example}[]
	Let $\Omega_2=\Omega_1=\mathbb{R}$ and set $(\gamma_n)_{n=0}^\infty$ to be the sequence of Hermite 
	functions. These
	functions are given by $\gamma_n(x) = a_n h_n(x) \exp(-x^2/2)$, where 
	$h_n$ is
	the $n$-th Hermite polynomial and $a_n$ is a normalization constant. Let 
	$f$ be a normalized Gaussian function centered around a
	point $a$, \ie, 
	$$f(x) = \frac{1}{{\pi}^{1/4}}
	\exp(-(x-a)^2/2),$$
	
	for $x \in \mathbb{R}$. Consider approximating $f$ in 
	the linear span of $(\gamma_n)_{n=0}^{N-1}$ for
	some $N \in \mathbb{N}$. Hermite functions
	form an orthonormal basis of $L^2(\mathbb{R})$. Therefore, the approximation error in
	$L^2(\mathbb{R})$ converges to zero as $N\to \infty$. 
	However, the
	approximation error in $L^2(\mathbb{R})$ becomes small for $N \gg \frac{e}{2}
	a^2$, where $e$ is Euler's number.This follows from standard results on Hermite function approximation combined with Stirling’s asymptotic formula for N!~\cite[Chapter III. Theorem 1.2.]{Lubich:QCMD}.
	In other words, to have a good approximation one requires a number of
	functions that depends nonlinearly on the center of the Gaussian $a$. 
	Consider
	now the modified basis $(C_h \gamma_n)_{n=0}^\infty$ where
	$h(x) =
	x-a$. Note that $f = \gamma_0 \circ h$, \ie, one needs only one
	function of the sequence $(C_h \gamma_n)_{n=0}^\infty$ to reproduce the 
	target function exactly.   
\end{example}

Composing orthonormal bases with linear
mappings is a common practice in computational 
mathematics~\cite{Gottlieb:SpectralMethods:1977}. Such mappings are often chosen
based on insight or formal analytical reasoning~\cite{Chou:APNM183:201}.
Composition with non-linear mappings are more scarce in the 
literature, since it is harder to identify effective nonlinear mappings
based on insight. However, given a class $\mathfrak{H}$ of adequate mappings, one can use numerical optimization techniques to
choose an optimal mapping $h^* \in \mathfrak{H}$ for a given problem. In general, one can end up with a mapping $h^*$ that leads to a
sequence $(C_{h^*} \gamma_n)_{n=0}^\infty$ that is not complete. Our results
demonstrate that this can be avoided by setting $\mathfrak{H}$ to be a class of
invertible mappings with bounded volume distortion.  

While there are several ways to construct such classes $\mathfrak{H}$, one popular approach is to use neural networks.  Indeed, there exist
many invertible neural-network architectures~\cite{Kobyzev:PAMI43:3964}, such as, \eg, invertible residual 
neural networks~\cite{Behrmann:ICML2019:573}. Here, the network $h$ can be written as a composition of a sequence of $M$ blocks, \ie, $h
= h_1 \circ h_2 \circ \ldots \circ h_M$, where each block $h_i$ is given by
$$
h_i(x) = x + \text{NN}(x),
$$
and $\text{NN}$ is a standard multi-layer perceptron. Requiring $\text{NN}$ to be Lipschitz with a
Lipschitz constant $L<1$ ensures that $h_i$ is a
$\bigl((1-L)-(1+L)\bigr)$-bi-Lipschitz. This, in turn, implies that Jacobian
determinant of $h$ is bounded from above and below. Such networks are often used to generate
complex probability measures from simple ones and are called \emph{normalizing flows}~\cite{Papamakarios:JMLR22:1}. Our results suggest that the 
same
flows can be used to generate problem-specific bases. The following numerical example illustrates the idea. 
\begin{example}[Numerical Evidence for Improved Approximation using Composition with a nonlinear Mapping] 
	\label{example:2}
	Consider
the univariate functions $f_1$ and $f_2$ given by
\begin{align}
\label{eq:target_function}
	f_1(x) &= \sin(2 |x|) \ \exp(-x^2/2) \nonumber,\\
f_2(x) &= \sin(2x) \ \exp(-x^4/2) 
\end{align}
and plotted in \Cref{fig:1}.

\begin{figure*}[h!]
    \vskip 0.2in
    \centering
    \includegraphics[width=\textwidth]{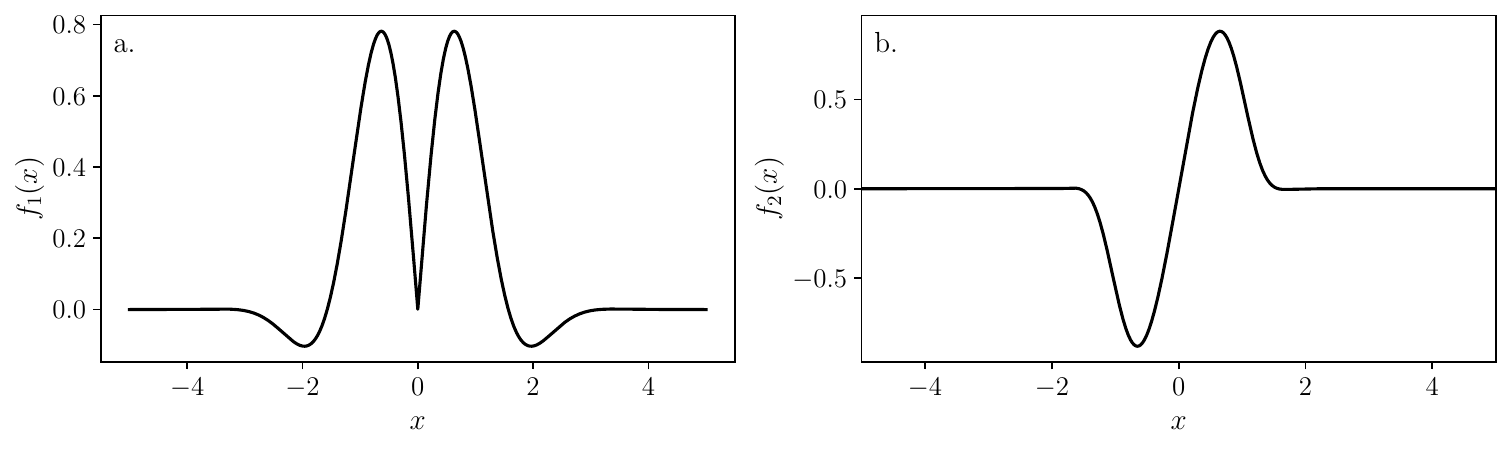}
    \caption{Plot of the functions $f_1$ (panel a) and $f_2$ (panel b) given by \eqref{eq:target_function}.}
   \label{fig:1}
   \vskip -0.2in
\end{figure*}

Clearly, the functions $f_1$ and $f_2$ belong to $L^2(\mathbb{R})$ and hence they can be
approximated in the linear span of Hermite functions $(\gamma_n)_{n=0}^{N-1}$
for some $N \in \mathbb{N}$. However, Hermite approximations of both functions
converge in a suboptimal manner. For $f_1$, this is due to the non-smooth behavior at $x=0$.
For $f_2$, this is due to its faster decay with respect to the weight
$\exp(-x^2/2)$ inherent to Hermite functions. 

Consider now
approximating $f_1$ and $f_2$ in the linear span of $(\gamma_n \circ
h_1)_{n=0}^{N-1}$ and $(\gamma_n \circ
h_2)_{n=0}^{N-1}$, respectively, where
$h_1$ and $h_2$ are bi-Lipschitz mapping that we modeled using a linear mapping composed with
an invertible residual neural
network of one block, \ie, $h_1$ and $h_2$ are given by 
\begin{align*}
h_1(x) &= \alpha_1 (x + \text{NN}_{\theta_1}(x)) + \beta_1, \\
h_2(x) &= \alpha_2 (x + \text{NN}_{\theta_2}(x)) + \beta_2
\end{align*}
where $\text{NN}_\theta$ denotes a multi-layer perceptron of one layer composed of 8
hidden units that use nonlinear activation functions. The parameters $\theta_1,
\theta_2$ along with the linear parameters $\alpha_1, \beta_1, \alpha_2, \beta_2
\in \mathbb{R}$ were
optimized to minimize the $L^2$-error in
approximating $f_1$ and $f_2$ in the linear span of $(\gamma_n \circ
h_1)_{n=0}^9$ and $(\gamma_n \circ h_2)_{n=0}^9$, respectively.
\Cref{fig:2} shows the perturbing functions $h_1$ and $h_2$. We estimated the
Lipschitz constants of $h_1$ and $h_2$ by computing the minimum and maximum of
$h'$ over a uniform grid of $5000$ points discretizing the domain $[-8, 8]$. The
Lipschitz constants are summarized in
\Cref{tab:lipschitz_constants}.

\begin{figure*}[h!]
    \vskip 0.2in
    \centering
    \includegraphics[width=0.3\textwidth]{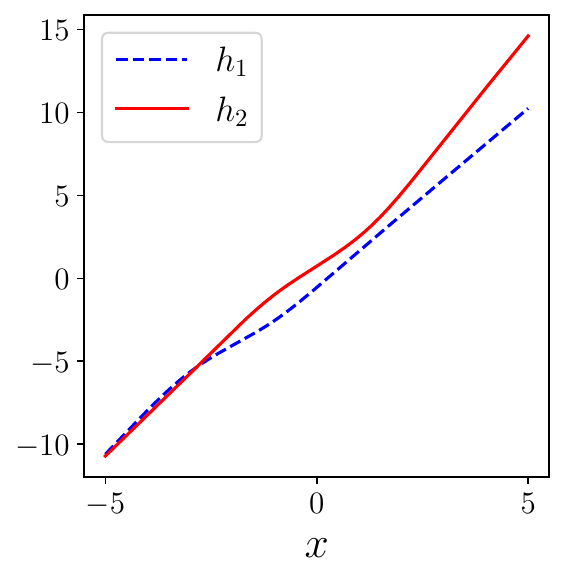}
    \caption{Plot of the bi-Lipschitz mappings $h_1$ and $h_2$ used to perturb Hermite functions for approximating the target functions $f_1$ and $f_2$.}
   \label{fig:2}
   \vskip 0.1in
\end{figure*}

\begin{table}[h!]
\centering
\caption{Estimated lower and upper Lipschitz constants for $h_1$ and $h_2$.}
\begin{tabular}{c|c|c}
\hline
Function & $L_{\text{lower}}$ & $L_{\text{upper}}$ \\
\hline
$h_1$ & 1.39 & 2.68 \\
$h_2$ & 1.61 & 3.22 \\
\hline
\end{tabular}
\label{tab:lipschitz_constants}
\end{table}
\end{example}

Following our theoretical analysis, $(\gamma_n \circ h_1)_{n=0}^{\infty}$ and $(\gamma_n \circ h_2)_{n=0}^{\infty}$ are
bases of $L^2(\mathbb{R})$ and the duals are given by $\bigl(\gamma_n \circ h_1 \
(h_1')\bigr)_{n=0}^{N-1}$, $\bigl(\gamma_n \circ h_2 \
(h_2')\bigr)_{n=0}^{N-1}$, respectively. Using these results, we computed the expansion
coefficients for multiple values of $N$ and compared the errors of the resulting
approximations with the errors of approximations obtained by using Hermite functions. \Cref{fig:3} shows the approximation error for different values of
$N$. The results show that more accurate approximations can be obtained with
our choice of the perturbing functions $h_1$ and $h_2$. Notably, the approximation of \( f_2 \) in the linear span of 
\( (\gamma_n \circ h_2)_{n=0}^{N-1} \) exhibits both higher accuracy 
and a faster convergence rate compared to the standard Hermite 
approximations. In contrast, for \( f_1 \), the convergence rates of 
the Hermite and perturbed Hermite approximations are comparable.  
\begin{figure*}[h!]
    \vskip 0.2in
    \centering
    \includegraphics[width=\textwidth]{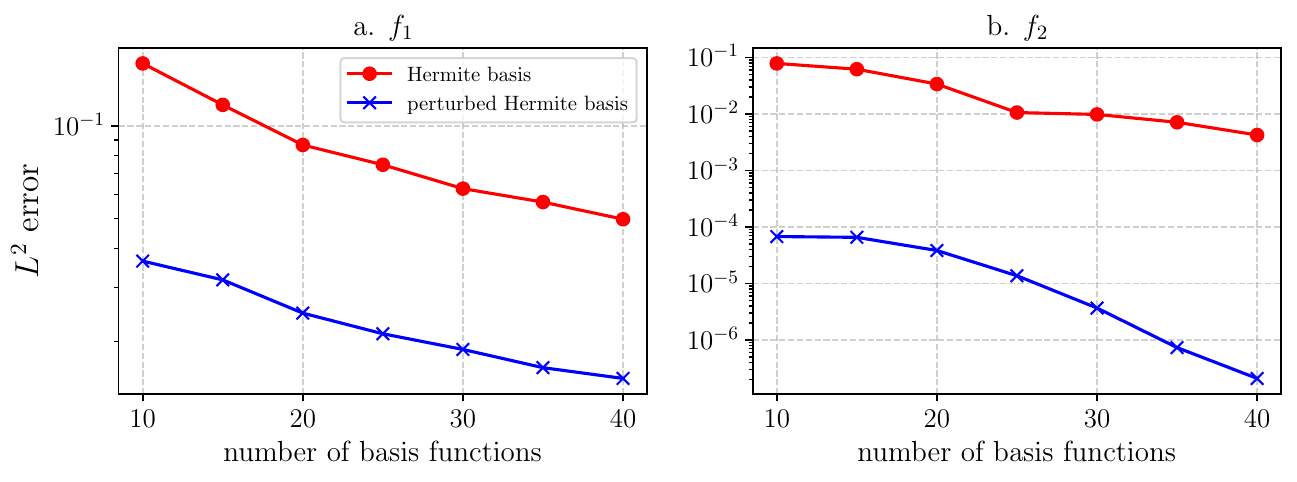}
    \caption{Convergence of the $L^2$-error in approximating the functions $f_1$
    (panel a) and $f_2$ (panel b)
   in the linear span of Hermite functions and the perturbed bases.}
   \label{fig:3}
   \vskip 0.1in
\end{figure*}

\Cref{fig:4} shows how Hermite basis functions transform under the
composition operators $C_{h_1}$ and $C_{h_2}$. Clearly, the perturbed basis
functions $(\gamma_n \circ h_1)_{n=0}^3$ exhibit
sharper behavior around the origin, at which the function $f_1$ is non-smooth.
Similarly, the functions $(\gamma_n \circ h_2)_{n=0}^3$ decay faster than
Hermite functions to match the fast decay of $f_2$.
\begin{figure*}[h!]%[H]
    \centering
    \includegraphics[width=\textwidth]{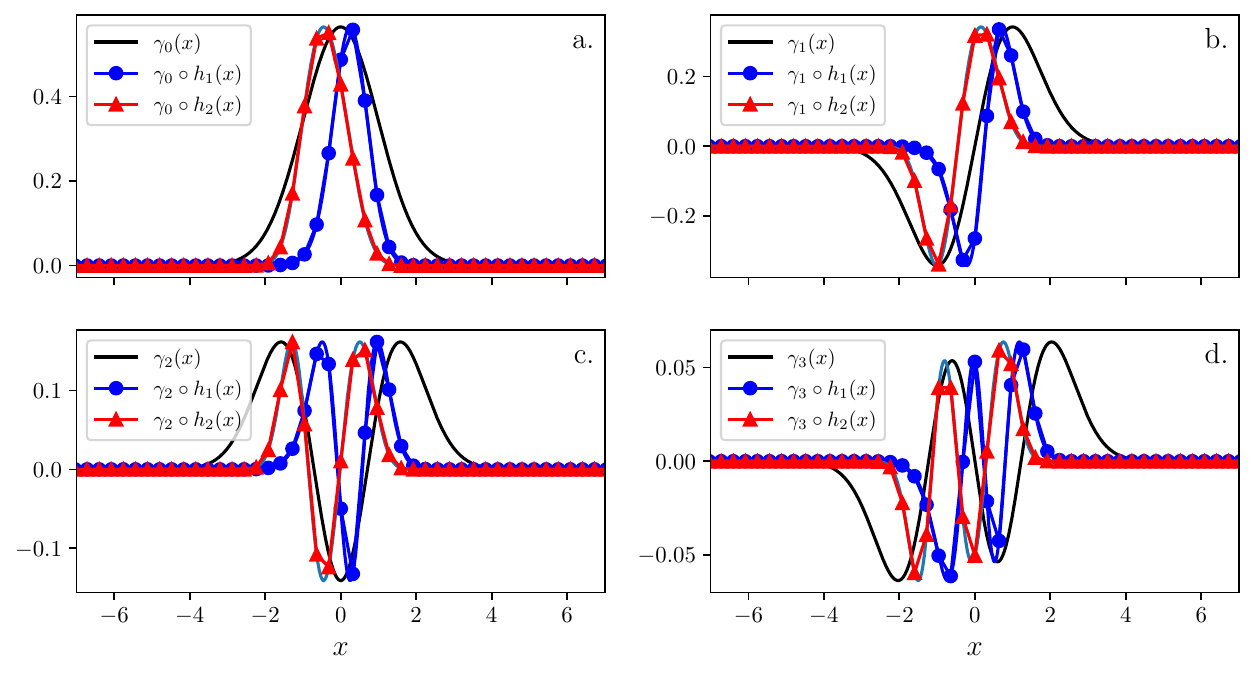}
    \caption{Plotted are the Hermite functions $(\gamma_n)_{n=0}^3$ (solid
    black lines), along with their perturbations 
$(\gamma_n \circ h_1)_{n=0}^3$ (red lines with triangle markers) and 
    $(\gamma_n \circ h_2)_{n=0}^3$ (blue lines with circle markers). The
    functions corresponding to $n=0, \ 1,\ 2,\ 3$ are plotted in panels a, b, c, d, respectively.}
    %} 
   \label{fig:4}
   \vskip 0.1in
\end{figure*}
\FloatBarrier

In summary, we showed in the previous examples that the approximation properties of
Hermite functions can be improved by composing them with a proper linear
or nonlinear
function of bounded volume distortion. This result is, indeed, not restricted to Hermite
functions and holds for other bases. Further, we demonstrated that the inducing mappings can be constructed
\emph{via} normalizing flows, \ie, invertible neural networks. While the examples are rather
simple, we note that a few
recent works have already made important first steps into more complex applications.
Specifically, \textcite{Cranmer:arXiv1904:05903} and
\textcite{Saleh:JCTC21:5221} utilized bi-Lipschitz normalizing flows to
optimize non-linear
perturbations to orthonormal bases of $L^2$ for solving differential equations.
\textcite{Saleh:JCTC21:5221} reported several orders-of-magnitude 
increased accuracy and improved convergence rates
upon optimizing a perturbation of the underlying basis. Our
results serve as a theoretical framework for such applications and open up new
avenues for the design and numerical analysis of problem-specific complete 
sequences.

\section*{Acknowledgments}
The authors would like to thank Eleonora Ficola, Stephan Wojtowytsch, 
Álvaro Fernández Corral and Emil Vogt for 
useful discussions and constructive feedback.

We acknowledge the support by the Deutsche Forschungsgemeinschaft (DFG) 
within the Research Training Group GRK 2583 ``Modeling, Simulation and Optimization of Fluid Dynamic Applications''. 

\section*{Author Contributions} Y. S. conceptualized the study, developed the
theory, wrote the manuscript and performed the numerical experiments. A.I.
contributed to the discussion and interpretation of results, and proofreading of the manuscript. 

\section*{Data Availability} The manuscript has no associated data. 

\section*{Code Availability} 

The codes we developed to produce the numerical results in \Cref{sec:summary}
are available at \url{https://github.com/robochimps/bases_by_composition}.
\section*{Declarations}

\textbf{Conflict of interest.} The authors declare that they have no conflict of interest.
\printbibliography%
%\bibliography{string,cmi}
\end{document}

%% Local Variables:
%% coding: utf-8
%% mode: LaTeX
%% mode: auto-fill
%% mode: flyspell
%% fill-column: 100
%% ispell-dictionary: "american"
%% reftex-cite-format: default
%% TeX-auto-save: t
%% TeX-close-quote: "''"
%% TeX-open-quote: "``"
%% TeX-parse-self: t
%% truncate-lines: t
%% End: